\documentclass[11pt]{article}
\usepackage{amsmath,amssymb,amsthm,mathtools}
\usepackage{microtype}
\usepackage{hyperref}

\newtheorem{theorem}{Theorem}[section]
\newtheorem{lemma}[theorem]{Lemma}

\theoremstyle{definition}

\newcommand{\Z}{\mathbb{Z}}
\newcommand{\R}{\mathbb{R}}
\newcommand{\ind}{\mathbf{1}}

\title{A proof of the Holmes--Holroyd--Ram\'irez cyclic-product conjecture}
\author{%
  Paata Ivanisvili\thanks{Department of Mathematics, University of California, Irvine, Irvine, CA 92697, USA.
  \texttt{pivanisv@uci.edu}.}
  \and
  Shengtong Zhang\thanks{Department of Mathematics, Stanford University, Stanford, CA 94305, USA.
  \texttt{stzh1555@stanford.edu}.}
}
\date{}

\begin{document}
\maketitle

\begin{abstract}
We prove the cyclic-product conjecture of Holmes, Holroyd and Ram\'irez: among all circular arrangements of a decreasing positive sequence, the circular-symmetric ``greedy'' arrangement maximises every consecutive cyclic product sum $P_r$.
\end{abstract}

\section{The conjecture}

Let $m\ge 2$ and let $a_0\ge a_1\ge\cdots\ge a_{m-1}>0$. For a circular arrangement $z=(z_0,\dots,z_{m-1})$ of these numbers, write
\[
P_r(z)=\sum_{k=0}^{m-1}\prod_{j=0}^{r-1}z_{k+j},
\qquad r\in\{1,\dots,m\},
\]
with indices read modulo $m$. The \emph{greedy} (circular-symmetric) arrangement $a^G$ is defined by
\begin{equation}\label{eq:greedy}
 a^G_i=a_{2i},\qquad a^G_{m-1-i}=a_{2i+1},
 \qquad i=0,\dots,\bigl\lfloor m/2\bigr\rfloor-1,
\end{equation}
with $a^G_{\lfloor m/2\rfloor}=a_{m-1}$ when $m$ is odd.

\begin{theorem}\label{thm:main}
For every $r\in\{1,\dots,m\}$ and every circular arrangement $z$ of $(a_0,\dots,a_{m-1})$,
\[
P_r(a^G)\ge P_r(z).
\]
\end{theorem}

The cases $r\in\{1,m\}$ are tautological. Holmes--Holroyd--Ram\'irez formulated this statement as Conjecture~5 and proved it for $r\le3$ and $r\ge m-3$ (hence for all $r$ when $m\le7$) \cite[Conjecture~5 and Theorem~6]{HHR}. The argument below treats all $m,r$ uniformly.

\section{A more general majorisation}

For $u\in\R^m$, write $u^\downarrow$ for the decreasing rearrangement of $u$. For $u,v\in\R^m$, we say that $u$ \emph{majorises} $v$, and write $u\succ v$, if
\[
\sum_{i=1}^k u_i^\downarrow\ge \sum_{i=1}^k v_i^\downarrow
\quad\text{for }k=1,\dots,m-1,
\qquad
\sum_{i=1}^m u_i=\sum_{i=1}^m v_i.
\]

Set $x_i=\log a_i$, so $x_0\ge\cdots\ge x_{m-1}$, and let $x^G_i=\log a^G_i$. For a circular arrangement $b=(b_0,\dots,b_{m-1})$ of $(x_0,\dots,x_{m-1})$, write
\[
S_k(b):=\sum_{j=0}^{r-1}b_{k+j},\qquad
S(b):=\bigl(S_0(b),\dots,S_{m-1}(b)\bigr),
\qquad
S^\downarrow(b):=S(b)^\downarrow.
\]
If $z_i=e^{b_i}$, then
\[
P_r(z)=\sum_k e^{S_k(b)},
\qquad
P_r(a^G)=\sum_k e^{S_k(x^G)}.
\]
Thus Theorem~\ref{thm:main} is the case $\varphi(u)=e^u$ of the following statement.

\begin{theorem}\label{thm:A}
For every $r\in\{1,\dots,m\}$ and every circular arrangement $b$ of $(x_0,\dots,x_{m-1})$,
\[
S^\downarrow(x^G)\succ S^\downarrow(b).
\]
Consequently, by Karamata's inequality,
\[
\sum_k\varphi\bigl(S_k(x^G)\bigr)
\ge
\sum_k\varphi\bigl(S_k(b)\bigr)
\]
for every convex $\varphi:\R\to\R$.
\end{theorem}

The rest of the note proves Theorem~\ref{thm:A}.

\section{Ky Fan sums and coverage}

For $s\in\{0,\dots,m\}$ write
\[
\Phi_{r,s}(b)=\text{sum of the $s$ largest coordinates of $S(b)$}.
\]
Theorem~\ref{thm:A} is equivalent to $\Phi_{r,s}(x^G)\ge\Phi_{r,s}(b)$ for all $s$, together with $\sum_kS_k=r\sum_ix_i$.

For functions $f,h:\Z/m\Z\to\R$, their circular convolution is
\[
(f*h)(t)=\sum_{u\in\Z/m\Z}f(u)h(t-u).
\]
Let $B\subset\Z/m\Z$ with $|B|=s$, and put $V_t=\{t,t-1,\dots,t-r+1\}$. Then
\[
\sum_{k\in B}S_k(b)=\sum_{t=0}^{m-1}c_B(t)b_t,
\]
where the \emph{coverage} $c_B:\Z/m\Z\to\Z$ is the circular convolution
\begin{equation}\label{eq:coverage}
 c_B(t)=|B\cap V_t|
 =\bigl(\ind_B*\ind_{\{0,\dots,r-1\}}\bigr)(t).
\end{equation}
Thus
\begin{equation}\label{eq:Phi}
\Phi_{r,s}(b)
=\max_{|B|=s}\langle c_B,b\rangle
\le\max_{|B|=s}\langle c_B^\downarrow,x^\downarrow\rangle,
\end{equation}
the last step being the classical rearrangement inequality.

\begin{lemma}\label{lem:lip}
For every $B\subset\Z/m\Z$ with $|B|=s$,
\[
|c_B(t+1)-c_B(t)|\le1,
\qquad
\mu\le c_B(t)\le M,
\qquad
\sum_tc_B(t)=sr,
\]
where $\mu=\max(0,s+r-m)$ and $M=\min(s,r)$.
\end{lemma}

\begin{proof}
The identity $c_B(t+1)-c_B(t)=\ind_B(t+1)-\ind_B(t-r+1)$ gives the Lipschitz bound. The upper bound is immediate, while
\[
|B\cap V_t|\ge|B|+|V_t|-m
\]
gives the lower bound. Finally, each element of $B$ lies in exactly $r$ of the windows $V_t$.
\end{proof}

For $s\in\{0,m\}$, the corresponding Ky Fan inequalities are equalities, while Theorem~\ref{thm:A} is immediate for $r\in\{1,m\}$. We therefore assume below that
\[
0<s<m,
\qquad
1<r<m.
\]

We now identify the candidate extremal coverage vector. By an $s$-arc we mean a set of $s$ consecutive residues in $\Z/m\Z$. Let $I\subset\Z/m\Z$ be an $s$-arc. Its coverage $c_I$ has the steepest profile permitted by the bounds and the Lipschitz condition in Lemma~\ref{lem:lip}: it is a circular trapezoid. Indeed, $c_I(t)=|I\cap V_t|$ is the overlap of a fixed $s$-arc with a moving $r$-arc. As the latter moves around the circle, the overlap increases by one at each step until one arc is contained in the other, remains maximal, then decreases by one at each step until the minimum overlap is reached. Thus $c_I$ increases with slope $+1$ from $\mu$ to $M$, has a plateau of length
\[
H=|s-r|+1
\]
at $M$, decreases with slope $-1$ to $\mu$, and has a plateau of length
\[
L=|m-s-r|+1
\]
at $\mu$. The upper plateau counts the positions in which the smaller arc lies inside the larger, while the lower plateau counts the positions for which the minimum overlap persists. The two ramps have length $M-\mu-1$ each, since
\[
M-\mu=\min(s,r,m-s,m-r).
\]
In particular $c_I$ attains both $\mu$ and $M$. Writing $\Delta=M-\mu$, the plateau lengths and the sum satisfy
\[
H+L+2(\Delta-1)=m,
\qquad
sr=m\mu+\Delta(H+\Delta-1).
\]
Thus $c_I$ is exactly the abstract steep $[\mu,M]$-trapezoid of sum $sr$ introduced below.

\begin{theorem}\label{thm:B}
If $I$ is an $s$-arc and $B$ is an arbitrary $s$-set, then $c_I\succ c_B$.
\end{theorem}

We first derive Theorem~\ref{thm:A} from Theorem~\ref{thm:B} and Lemma~\ref{lem:comono}. Sections~5 and~6 then prove Theorem~\ref{thm:B} through a more general Lipschitz-isoperimetric statement.

\begin{lemma}\label{lem:pair}
If $u\succ v$ in $\R^m$ and $y\in\R^m$ is decreasing, then
\[
\langle u^\downarrow,y\rangle\ge\langle v^\downarrow,y\rangle.
\]
\end{lemma}

\begin{proof}
Summation by parts gives
\[
\sum_{k=1}^my_ku_k^\downarrow
=y_m\sum_{k=1}^mu_k^\downarrow
+\sum_{j=1}^{m-1}(y_j-y_{j+1})\sum_{k=1}^ju_k^\downarrow.
\]
The total sums of $u$ and $v$ agree, each coefficient $y_j-y_{j+1}$ is nonnegative, and the partial sums of $u^\downarrow$ dominate those of $v^\downarrow$.
\end{proof}

Choose $I$ as in Lemma~\ref{lem:comono}. Granting Theorem~\ref{thm:B} and Lemmas~\ref{lem:pair} and~\ref{lem:comono},
\[
\begin{aligned}
\Phi_{r,s}(b)
&=\max_{|B|=s}\langle c_B,b\rangle\\
&\le\max_{|B|=s}\langle c_B^\downarrow,x^\downarrow\rangle\\
&\le\langle c_I^\downarrow,x^\downarrow\rangle
&&\text{by Theorem~\ref{thm:B} and Lemma~\ref{lem:pair}}\\
&=\langle c_I,x^G\rangle
&&\text{by Lemma~\ref{lem:comono}}\\
&\le\Phi_{r,s}(x^G).
\end{aligned}
\]
This proves Theorem~\ref{thm:A}.

\section{Comonotonicity}\label{sec:comono}

For $c\in\frac12\Z/m\Z$ and $t\in\Z/m\Z$, let
\[
d_c(t)=\min_{k\in\Z}|t-c+km|
\]
be the circular distance from $t$ to $c$, using any representatives of $t$ and $c$. A function $f:\Z/m\Z\to\R$ is \emph{circularly symmetrically decreasing about the centre $c$} if $f(t)$ depends only on $d_c(t)$ and is nonincreasing as $d_c(t)$ increases. If $c$ is a half-integer, the centre is the seam between two adjacent residues.

\begin{lemma}[Convolution preserves circular symmetry]\label{lem:conv-sym}
If $f,h:\Z/m\Z\to[0,\infty)$ are circularly symmetrically decreasing about centres $c_f$ and $c_h$, respectively, then $f*h$ is circularly symmetrically decreasing about $c_f+c_h$.
\end{lemma}

\begin{proof}
Such a function is a nonnegative linear combination of indicators of arcs with the same centre, with coefficients given by the successive drops on the distance-rings. If $A$ and $C$ are arcs centred at $c_A$ and $c_C$, respectively, then
\[
(\ind_A*\ind_C)(t)=|A\cap(t-C)|.
\]
As $t$ moves away from $c_A+c_C$, this overlap is nonincreasing and changes by at most one at each step. Its profile is therefore a possibly degenerate circular trapezoid, circularly symmetrically decreasing about $c_A+c_C$. The claim follows by bilinearity.
\end{proof}

\begin{lemma}[Comonotonicity]\label{lem:comono}
Place $x^G$ according to \eqref{eq:greedy}. Let
\[
\sigma=-\Bigl\lfloor\frac{r+s-1}{2}\Bigr\rfloor\pmod m,
\qquad
I=\{\sigma,\sigma+1,\dots,\sigma+s-1\},
\]
where the elements of $I$ are read modulo $m$.
Then $x^G$ and $c_I$ are comonotone:
\[
(c_I(p)-c_I(q))(x^G_p-x^G_q)\ge0
\qquad\text{for all }p,q.
\]
Consequently
\[
\langle c_I,x^G\rangle=\langle c_I^\downarrow,x^\downarrow\rangle.
\]
\end{lemma}

\begin{proof}
The indicator $\ind_I$ is circularly symmetrically decreasing about $\sigma+(s-1)/2$, and $\ind_{\{0,\dots,r-1\}}$ is circularly symmetrically decreasing about $(r-1)/2$. Lemma~\ref{lem:conv-sym} puts the centre of $c_I$ at
\[
\sigma+\frac{s+r-2}{2}
=
\begin{cases}
-1/2,&r+s\text{ odd},\\
0,&r+s\text{ even}.
\end{cases}
\]
In the first case
\[
c_I(0)=c_I(m-1)\ge c_I(1)=c_I(m-2)\ge\cdots,
\]
while in the second case $c_I(d)=c_I(-d)$ and $c_I$ is nonincreasing in the distance from $0$. Thus $c_I$ is nonincreasing along
\[
0,\;m-1,\;1,\;m-2,\;2,\;\dots,
\]
the order in which \eqref{eq:greedy} places the decreasing sequence $x$. Hence $x^G$ and $c_I$ are comonotone.
\end{proof}

The next two sections complete the proof of Theorem~\ref{thm:B}.

\section{Superlevel sizes}

A \emph{steep circular $[\mu,M]$-trapezoid} with plateau lengths $H,L\ge1$ is a circular sequence which, up to rotation and reversal, consists of a plateau of $H$ copies of $M$, a unit-slope descent through $M-1,\dots,\mu+1$, a plateau of $L$ copies of $\mu$, and a unit-slope ascent through $\mu+1,\dots,M-1$.

For the remainder of the proof, $\mu<M$ are integers, $\Delta=M-\mu$, and $S$ is such that the steep $[\mu,M]$-trapezoid of sum $S$ exists. Writing $Q=S-m\mu$, this means
\[
H=\frac{Q}{\Delta}-\Delta+1,
\qquad
L=m-H-2(\Delta-1)
\]
are integers with $H,L\ge1$. Its decreasing rearrangement is
\[
\tau^\downarrow
=\bigl(
\underbrace{M,\dots,M}_{H},
\underbrace{M-1,M-1}_{2},\dots,
\underbrace{\mu+1,\mu+1}_{2},
\underbrace{\mu,\dots,\mu}_{L}
\bigr),
\]
and its superlevel sizes are
\begin{equation}\label{eq:n-tau}
 n_j(\tau)=|\{t:\tau(t)\ge\mu+j\}|
 =H+2(\Delta-j),
 \qquad j=1,\dots,\Delta.
\end{equation}

For an integer sequence $v:\Z/m\Z\to[\mu,M]$, set
\[
n_j(v)=|\{t:v(t)\ge\mu+j\}|,
\qquad j=1,\dots,\Delta,
\]
and write
\[
n(v)=\bigl(n_1(v),\dots,n_\Delta(v)\bigr).
\]
For $k\in\{0,\dots,m\}$, the layer-cake formula gives
\begin{equation}\label{eq:layer}
\sum_{i=1}^kv_i^\downarrow
=k\mu+\sum_{j=1}^\Delta\min(k,n_j(v)).
\end{equation}
Thus $\tau\succ v$ precisely when the right-hand side for $\tau$ dominates that for $v$ for every $k$, together with equality of the total sums.

\begin{lemma}\label{lem:clip}
Let $x,y\in\R^\Delta$ be decreasing and satisfy $x\succ y$. Then, for every $k\ge0$,
\[
\sum_{i=1}^\Delta\min(k,x_i)
\le
\sum_{i=1}^\Delta\min(k,y_i).
\]
\end{lemma}

\begin{proof}
The function $u\mapsto\min\{k,u\}$ is concave. Since $x\succ y$, Karamata's inequality gives
\[
\sum_{i=1}^\Delta\min\{k,x_i\}
\le
\sum_{i=1}^\Delta\min\{k,y_i\}.
\]
\end{proof}

\begin{lemma}\label{lem:full-amp}
Let $v$ be a circular Lipschitz-$1$ integer sequence with values in $[\mu,M]$, attaining both endpoints, and satisfying $\sum v=S$. Then
\[
n(v)\succ n(\tau),
\qquad\text{and consequently}\qquad
\tau\succ v.
\]
\end{lemma}

\begin{proof}
Choose $p,q$ with $v(p)=M$ and $v(q)=\mu$. The two arcs joining $p$ to $q$ are internally disjoint. Along each arc, integer-valuedness and the Lipschitz-$1$ condition force every intermediate value. Hence
\[
n_j(v)-n_{j+1}(v)=|\{v=\mu+j\}|\ge2,
\qquad j=1,\dots,\Delta-1.
\]
The sequence \eqref{eq:n-tau} has consecutive differences exactly $2$, so
\[
d_j=n_j(v)-n_j(\tau)
\]
is nonincreasing. Moreover,
\[
\sum_{j=1}^\Delta n_j(v)
=\sum_t\bigl(v(t)-\mu\bigr)
=S-m\mu
=Q
=\sum_{j=1}^\Delta n_j(\tau),
\]
so $\sum_{j=1}^\Delta d_j=0$. Hence the partial sums of $(d_j)$ are nonnegative, and therefore $n(v)\succ n(\tau)$. Lemma~\ref{lem:clip} and \eqref{eq:layer} give $\tau\succ v$.
\end{proof}

\section{Smaller boxes}

\begin{lemma}[Padded comparison]\label{lem:padded}
Let $\alpha,\beta$ be nonnegative integers and $\Delta'=\Delta-\alpha-\beta\ge1$. Write
\[
Q'=Q-m\alpha,
\qquad
H'=\frac{Q'}{\Delta'}-\Delta'+1,
\]
and define
\[
n^\sharp
=\bigl(
\underbrace{m,\dots,m}_{\alpha},
\bigl(H'+2(\Delta'-j)\bigr)_{j=1}^{\Delta'},
\underbrace{0,\dots,0}_{\beta}
\bigr)\in\R^\Delta.
\]
If $H'\ge0$ and $H'+2(\Delta'-1)\le m$, then $n^\sharp\succ n(\tau)$.
\end{lemma}

\begin{proof}
Both vectors are decreasing and sum to $Q$. Write $P^\sharp(p)$ and $P^\tau(p)$ for the sums of their first $p$ coordinates. If $p\le \alpha$, then
\[
P^\sharp(p)-P^\tau(p)=p(L+p-1)\ge0.
\]
If $p\ge \alpha+\Delta'$, then $P^\sharp(p)=Q\ge P^\tau(p)$; this also handles $\Delta'=1$.

Suppose $\Delta'\ge2$ and $\alpha<p<\alpha+\Delta'$. Put $t=p-\alpha$. Then
\begin{align*}
P^\sharp(p)
&=\alpha m+\frac{t(Q-m\alpha)}{\Delta'}+t\Delta'-t^2,\\
P^\tau(p)
&=(\alpha+t)\left(\frac{Q}{\Delta}+\Delta-\alpha-t\right).
\end{align*}
Using $Q=\Delta(H+\Delta-1)$ and $m=H+L+2\Delta-2$, a direct expansion gives
\begin{equation}\label{eq:E}
\Delta'\bigl(P^\sharp(p)-P^\tau(p)\bigr)
=Ht\beta+L\alpha(\Delta'-t)+t\beta(\beta-1)+\alpha(1-\alpha)(t-\Delta').
\end{equation}
Every term on the right is nonnegative: the first two are immediate, $\beta(\beta-1)\ge0$ for integer $\beta\ge0$, and the last term vanishes for $\alpha\in\{0,1\}$ and is nonnegative for $\alpha\ge2$. Hence $P^\sharp(p)\ge P^\tau(p)$.
\end{proof}

\begin{lemma}[Smaller box]\label{lem:smaller-box}
Let $\tau$ be the steep $[\mu,M]$-trapezoid of sum $S$, and let $v$ be a circular Lipschitz-$1$ integer sequence with the same sum. Set
\[
\mu'=\min_tv(t),
\qquad
M'=\max_tv(t).
\]
If $[\mu',M']\subsetneq[\mu,M]$, then $\tau\succ v$.
\end{lemma}

\begin{proof}
Let $\Delta'=M'-\mu'$, $\alpha=\mu'-\mu$, $\beta=M-M'$, and $Q'=S-m\mu'$. If $\Delta'=0$, then $v$ is constant, and every vector with the same sum majorises $v$.

Assume $\Delta'\ge1$ and set
\[
n_j^{\mathrm{loc}}(v)=|\{t:v(t)\ge\mu'+j\}|,
\qquad j=1,\dots,\Delta'.
\]
The two-arc argument in Lemma~\ref{lem:full-amp} gives
\[
n_j^{\mathrm{loc}}(v)-n_{j+1}^{\mathrm{loc}}(v)\ge2,
\qquad j=1,\dots,\Delta'-1.
\]
Since $v$ attains $\mu'$ and $M'$, we have $n_{\Delta'}^{\mathrm{loc}}(v)\ge1$ and $n_1^{\mathrm{loc}}(v)\le m-1$. Therefore, for $j=1,\dots,\Delta'$,
\[
1+2(\Delta'-j)
\le n_j^{\mathrm{loc}}(v)
\le m-1-2(j-1).
\]
Summing these inequalities over $j$ gives
\[
\Delta'^2\le Q'=\sum_{j=1}^{\Delta'}n_j^{\mathrm{loc}}(v)
\le\Delta'(m-\Delta').
\]
Thus, for
\[
H'=\frac{Q'}{\Delta'}-\Delta'+1,
\]
we have $H'\ge1$ and $H'+2(\Delta'-1)\le m-1$.

Let
\[
n^\sharp_{\mathrm{loc}}=(H'+2(\Delta'-1),\dots,H')
\]
and set
\[
d_j=n_j^{\mathrm{loc}}(v)-\bigl(H'+2(\Delta'-j)\bigr),
\qquad j=1,\dots,\Delta'.
\]
Then
\[
d_j-d_{j+1}
=n_j^{\mathrm{loc}}(v)-n_{j+1}^{\mathrm{loc}}(v)-2
\ge0,
\qquad j=1,\dots,\Delta'-1,
\]
so $d_1\ge\cdots\ge d_{\Delta'}$. Moreover,
\[
\sum_{j=1}^{\Delta'}d_j
=Q'-\bigl(\Delta'H'+\Delta'(\Delta'-1)\bigr)
=0
\]
by the definition of $H'$. A decreasing sequence with zero sum has nonnegative partial sums, and hence
\[
n^{\mathrm{loc}}(v)\succ n^\sharp_{\mathrm{loc}}.
\]
Define
\[
\widetilde n(v)
=\bigl(
\underbrace{m,\dots,m}_{\alpha},
 n_1^{\mathrm{loc}}(v),\dots,n_{\Delta'}^{\mathrm{loc}}(v),
\underbrace{0,\dots,0}_{\beta}
\bigr).
\]
Since both local vectors are decreasing and take values in $[0,m]$, adjoining the same initial block of $m$'s and terminal block of $0$'s preserves majorisation. Therefore Lemmas~\ref{lem:clip} and~\ref{lem:padded} give, for every $k$,
\[
\sum_{j=1}^\Delta\min(k,\widetilde n_j(v))
\le
\sum_{j=1}^\Delta\min(k,n_j^\sharp)
\le
\sum_{j=1}^\Delta\min(k,n_j(\tau)).
\]
The vector $\widetilde n(v)$ is the superlevel vector of $v$ measured from the baseline $\mu$. By \eqref{eq:layer}, this is $\tau\succ v$.
\end{proof}

\begin{theorem}[Lipschitz isoperimetry]\label{thm:C}
Let $\mu<M$ be integers, and let $\tau$ be the steep circular trapezoid with values in $[\mu,M]$, plateaus only at $\mu$ and $M$, and sum $S$. Then $\tau\succ v$ for every circular Lipschitz-$1$ integer sequence $v$ with values in $[\mu,M]$ and sum $S$.
\end{theorem}

\begin{proof}
If $v$ attains both endpoints, apply Lemma~\ref{lem:full-amp}. Otherwise its actual range $[\min v,\max v]$ is a strict subbox of $[\mu,M]$, so Lemma~\ref{lem:smaller-box} applies.
\end{proof}

\begin{proof}[Proof of Theorem~\ref{thm:B}]
Lemma~\ref{lem:lip} shows that $c_B$ is a circular Lipschitz-$1$ integer sequence with values in $[\mu,M]$ and sum $sr$. The identities displayed after the description of $c_I$ show that $c_I$ is exactly the steep circular $[\mu,M]$-trapezoid of sum $sr$. Hence Theorem~\ref{thm:C} gives $c_I\succ c_B$.
\end{proof}

This completes the proof of Theorem~\ref{thm:B}, hence of Theorem~\ref{thm:A}, and therefore of Theorem~\ref{thm:main}.

\section*{Acknowledgements}

P.I.\ acknowledges partial support from NSF CAREER grant DMS-2152401,
NSF grant DMS-2554183, a Simons Fellowship, and a Humboldt Research Fellowship
for Experienced Researchers.
 The authors subsequently checked all mathematical arguments in detail and edited and refined the resulting proof for inclusion in the final manuscript. The authors take full responsibility for the correctness of all statements and proofs appearing in the paper. A complete transcript of the conversation with Grok is available at \url{https://grok.com/share/c2hhcmQtNA_73c57e1f-2ff7-467b-a325-250dcd6b9e96}.

\end{document}